\documentclass[11pt]{article}
\usepackage{caption}
\usepackage{epsf,epsfig,amsfonts,amsgen,amsmath,amstext,amsbsy,amsopn,amsthm
}
\usepackage{ebezier,eepic}
\usepackage{color}
\usepackage{multirow}
\usepackage{mathrsfs}
\usepackage{tikz}
\usepackage{enumerate}
\usepackage{enumitem}
\usepackage{url}
\usepackage{pgf,tikz,pgfplots}
\usepackage{mathrsfs}
\usepackage{amssymb}

\usetikzlibrary{arrows}

\newtheorem{dfn}{Definition}[section]
\newtheorem{thm}[dfn]{Theorem}
\newtheorem{lem}[dfn]{Lemma}

\def\CN{\mathrm{CN}}
\def\ex{\mathrm{ex}}

\let\svthefootnote\thefootnote
\newcommand\blankfootnote[1]{%
	\let\thefootnote\relax\footnotetext{#1}%
	\let\thefootnote\svthefootnote%
}

\begin{document}

\title{%Bipartite Tur\'an problem in odd uniform hypergraph}
A hypergraph bipartite Tur\'an problem with odd uniformity}

\author{
Jie Ma\thanks{School of Mathematical Sciences, University of Science and Technology of China, Hefei, Anhui, 230026, China.
Research supported by National Key Research and Development Program of China 2023YFA1010201, National Natural Science Foundation of China grant 12125106, and Anhui Initiative in Quantum Information Technologies grant AHY150200. Email: jiema@ustc.edu.cn.}
~~~~
Tianchi Yang\thanks{Department of Mathematics, National University of Singapore, 119076, Singapore.
Research supported by Professor Hao Huang's start-up grant at NUS and an MOE Academic Research Fund (AcRF) Tier 1 grant. Email: tcyang@nus.edu.sg.
}
}

\date{}

%\blankfootnote{ fgfdsg}

\maketitle
\begin{abstract}
In this paper, we investigate the hypergraph Tur\'an number $\ex(n,K^{(r)}_{s,t})$.
Here, $K^{(r)}_{s,t}$ denotes the $r$-uniform hypergraph with vertex set $\left(\cup_{i\in [t]}X_i\right)\cup Y$ and edge set $\{X_i\cup \{y\}: i\in [t], y\in Y\}$, where $X_1,X_2,\cdots,X_t$ are $t$ pairwise disjoint sets of size $r-1$ and $Y$ is a set of size $s$ disjoint from each $X_i$.
This study was initially explored by Erd\H{o}s and has since received substantial attention in research.
Recent advancements by Brada\v{c}, Gishboliner, Janzer and Sudakov have greatly contributed to a better understanding of this problem.
They proved that $\ex(n,K_{s,t}^{(r)})=O_{s,t}(n^{r-\frac{1}{s-1}})$ holds for any $r\geq 3$ and $s,t\geq 2$.
They also provided constructions illustrating the tightness of this bound if $r\geq 4$ is {\it even} and $t\gg s\geq 2$.
Furthermore, they proved that $\ex(n,K_{s,t}^{(3)})=O_{s,t}(n^{3-\frac{1}{s-1}-\varepsilon_s})$ holds for $s\geq 3$ and some $\epsilon_s>0$.
Addressing this intriguing discrepancy between the behavior of this number for $r=3$ and the even cases,
Brada\v{c} et al. post a question of whether  
\begin{equation*}
\mbox{$\ex(n,K_{s,t}^{(r)})= O_{r,s,t}(n^{r-\frac{1}{s-1}- \varepsilon})$ holds for {\it odd} $r\geq 5$ and any $s\geq 3$.}
\end{equation*}

In this paper, we provide an affirmative answer to this question, utilizing novel techniques to identify regular and dense substructures.
This result highlights a rare instance in hypergraph Tur\'an problems where the solution depends on the parity of the uniformity.
\end{abstract}

\section{Introduction}
For a given $r$-uniform hypergraph $H$, we say an $r$-uniform hypergraph is {\it $H$-free} if it does not contain a copy of $H$ as its subgraph.
The {\it Tur\'an number} $\ex(n,H)$ denotes the maximum number of edges in an $H$-free $r$-uniform hypergraph on $n$ vertices.
The study of Tur\'an number is a central problem in extremal combinatorics.
We refer interested readers to the survey by F\"uredi and Simonovits \cite{FS-survey} for ordinary graphs and
the survey by Keevash \cite{Kv11} for non-$r$-partite $r$-uniform hypergraphs.
Here, our focus lies on the Tur\'an numbers of $r$-partite $r$-uniform hypergraphs $H$ for $r\geq 3$.
A fundamental result proved by Erd\H{o}s states that for every such $H$,
$\ex(n,H)=O(n^{r-\varepsilon_H})$ holds for some $\varepsilon_H>0$.
The primary objective of this aspect is to determine the optimal constant $\varepsilon_H$.
However, this problem is notoriously difficult and to date, there are very few cases that have been fully understood.

In this paper, we consider the Tur\'an number of the following $r$-partite $r$-uniform hypergraphs,
which were initially defined by Mubayi and Verstra\"ete \cite{MuVe04}:
for positive integers $r,s,t$, let $K^{(r)}_{s,t}$ denote the $r$-uniform hypergraph with vertex set $\left(\cup_{i\in [t]}X_i\right)\cup Y$ and edge set $\{X_i\cup \{y\}: i\in [t], y\in Y\}$,
where $X_1,X_2,\cdots,X_t$ are $t$ pairwise disjoint sets of size $r-1$ and $Y$ is a set of size $s$ disjoint from each $X_i$.
This study can be traced back to an old problem posted by Erd\H{o}s \cite{E77},
who asked to determine the maximum number $f_r(n)$ of edges in an $r$-uniform hypergraph on $n$ vertices that does not contain four distinct edges
$A, B, C, D$ satisfying $A\cup B=C\cup D$ and $A\cap B=C\cap D=\emptyset$.
This generalizes the Tur\'an number of the four-cycle, and it is evident to see that $f_3(n)=\ex(n,K^{(3)}_{2,2})$ and $f_r(n)\le \ex(n,K^{(r)}_{2,2})$ for any $r\geq 4$.
F\"uredi \cite{Fu84} resolved a conjecture of Erd\H{o}s made in \cite{E77}, by showing that $f_r(n)\leq 3.5\binom{n}{r-1}$ for any $r\geq 3$.
This was first improved by Mubayi and Verstra\"ete \cite{MuVe04}, and later on,
further improvements were made by Pikhurko and Verstra\"ete \cite{PiVe09}.

Returning to the Tur\'an number $\ex(n,K^{(r)}_{s,t})$,
Mubayi and Verstra\"ete \cite{MuVe04} primarily focus on the case $r=3$.
They proved that $\ex(n,K^{(3)}_{s,t})=O_{s,t}(n^{3-1/s})$ for $t\ge s\ge 3$ and $\ex(n,K^{(3)}_{s,t})=\Omega_t(n^{3-2/s})$ for $t>(s-1)!$.
For the particular case $s=2$, Mubayi and Verstra\"ete \cite{MuVe04} also provided that $\ex(n,K^{(3)}_{2,t})\leq t^4\binom{n}{2}$ for $t\ge 3$, and they further posed the question of determining the order of the magnitude of the leading coefficient in terms of $t$.  
Among other results, Ergemlidze, Jiang and Methuku \cite{EJM20} obtained an improvement by showing $\ex(n,K^{(3)}_{2,t})\leq (15t\log t+40t)\binom{n}{2}$, which can be extended to all $r\geq 3$.
Using the random algebraic method (see \cite{Bukh15,BD}), Xu, Zhang and Ge \cite{XZG20,XZG21} proved that $\ex(n,K^{(r)}_{s,t})=\Theta(n^{r-1/t})$, assuming that $s$ is sufficiently large than $r,t$.

Very recently, Brada\v{c}, Gishboliner, Janzer and Sudakov \cite{BGJS23} made significant contributions towards
a better understanding of the behavior of the Tur\'an number $\ex(n,K^{(r)}_{s,t})$.
Using a novel variant of the dependent random choice,
they proved a general upper bound that $\ex(n,K^{(r)}_{s,t})=O_s\left(t^{\frac{1}{s-1}}n^{r-\frac{1}{s-1}}\right)$ holds for any $r\geq 3$ and $s,t\geq 2$.
Moreover, they built upon norm graphs (\cite{ARS99,KRS96}) and provided matching constructions, which led to $$\ex(n,K^{(r)}_{s,t})=\Theta_{r,s}\left(t^{\frac{1}{s-1}}n^{r-\frac{1}{s-1}}\right)\mbox{ for any {\it even} $r\geq 4$, $s\geq 2$, and $t>(s-1)!$.}$$
%This provides solutions to multiple problems raised in \cite{MuVe04,EJM20,XZG21}.
Furthermore, they derived a different order of magnitude for $n$ in the case $r=3$ by proving that
$$\ex(n,K_{s,t}^{(3)})=O_{s,t}\left(n^{3-\frac{1}{s-1}- \varepsilon_s}\right) \mbox{ holds for any $s\ge 3$, $t$, and some positive constant $\varepsilon_s=O(s^{-5})$}.$$
Brada\v{c} et al. posed the question of whether the above upper bound can be extended to all odd uniformities.
They noted that if the question is affirmative, then ``this would be a rare example of an extremal problem where the answer depends on the parity of the uniformity'', quoted from \cite{BGJS23}.

In this work, we provide a positive answer to the aforementioned question posed by Brada\v{c} et al. Our main result can be stated as follows.

\begin{thm}\label{thm: main}
For any odd $r\geq 3$ and any $s\geq 3$, there exists some $\varepsilon=\varepsilon(s)>0$ depending only on $s$ such that for any positive integer $t$, $$\ex(n,K_{s,t}^{(r)})=O_{r,s,t}\left(n^{r-\frac{1}{s-1}- \varepsilon}\right).$$
\end{thm}

We use a different proof approach from \cite{BGJS23} (see Section~\ref{subsec:sketch} for an outline of the proof).
The core ideas are to find some regular and dense substructures in hypergraphs.  
Our proof works for $\varepsilon(s)= \frac{1}{6(s+2)^2}$, although we have made no serious attempt to optimize the leading coefficient.
This improves the choice of $\varepsilon_s=O(s^{-5})$ in \cite{BGJS23} for the case $r=3$.

\medskip

The remainder of the paper is organized as follows.
In Section 2, we introduce the necessary notation and provide an outline of the proof for Theorem \ref{thm: main}.
In Section 3, we break down the proof of Theorem \ref{thm: main} into three lemmas.
The full proofs of these lemmas are presented in Section 4.
Finally, in Section 5, we offer some concluding remarks.

\section{Preliminaries}
In this section, we begin by introducing the necessary notation, followed by providing a preliminary outline of the proof for our main result, namely Theorem~\ref{thm: main}.

\subsection{Notation}
Let $r\ge 3$, $s\ge 3$, $t$ and $n$ be positive integers throughout the rest of this paper. Let $[n]=\{1,2,\cdots,n\}$.

Assume that $\mathcal G$ is an $r$-partite $r$-uniform hypergraph with parts $V_1, \cdots , V_r$, each of size $n$, throughout this section.
For a given vertex $v$, the {\it link hypergraph} of $v$, denote as $N_{\mathcal G}(v) $, comprises all $(r-1)$-sets that, when combined with $v$, form an edge in $\mathcal G$.
Similarly, for $k<r$ and a set $T$ with $k$ vertices, we write $N_{\mathcal G}(T)$ for the $(r-k)$-uniform hypergraph containing all $(r-k)$-sets which together with $T$ form an edge in $\mathcal G$.
We denote its cardinality as $d_{\mathcal G}(T)=| N_{\mathcal G}(T)|$.
Specially if $ d_{\mathcal G}(T)\neq 0$, we say the $k$-set $T$ is a {\it $k$-tuple} of $\mathcal G$.
We define the set of all $(r-1)$-tuples of $\mathcal G$ as $\mathcal T(\mathcal G)$.
For $i\in [r]$, let $\mathcal T_i(\mathcal G)$ be the set of $(r-1)$-tuples of $\mathcal G$ contained in $V(\mathcal G)\backslash V_i$.

For an $s$-set $S=\{v_1,v_2\cdots,v_{s} \} \subseteq V(\mathcal G)$, we define $\CN_{\mathcal G}(S) $ as the set of common edges in all link hypergraphs of vertices $v_i\in S$. That is, $\CN_{\mathcal G}(S)=\bigcap_{i\in [s]} N_{\mathcal G}(v_i)$.
A {\it vertex cover} of a set $\mathcal{E}$ of edges is a set of vertices that intersects with every edge in $\mathcal{E}$.
For each $s$-set $S$, we choose and fix a minimum vertex-cover of $\CN_{\mathcal G}(S)$, and call these vertices the {\it roots} of $S$.
For a root $v$ of $S$, we also say $S$ is rooted on $v$.
The following simple yet crucial property will be repeatedly used in the proofs:
\begin{equation}\label{equ:roots}
\mbox{If $\mathcal G$ is $K_{s,t}^{(r)}$-free, then any $s$-set $S\subseteq V(\mathcal{G})$ has less than $rt$ roots.}
\end{equation}
To see this, consider a maximum set of disjoint edges in $\CN_{\mathcal G}(S)$, and let $A$ be the vertex set of these edges.
Due to the maximality, every edge in $\CN_{\mathcal G}(S)$ contains at least one vertex in $A$. So $A$ is a vertex-cover of $\CN_{\mathcal G}(S)$.
Since $\mathcal G$ is $K_{s,t}^{(r)}$-free, $\CN_{\mathcal G}(S)$ has at most $t-1$ disjoint edges. Therefore, $|A|\le (t-1)(r-1)<tr$, as desired.

Let $S$ be an $s$-set in $V(\mathcal G)$.
We denote $cd_{\mathcal G}(S)=|\CN_{\mathcal G}(S)|$ to be the {\it codegree} of $S$ in $\mathcal G$.
For a vertex $u\notin S$, we write $cd_{\mathcal G}(S|u)$ for the number of edges in $\CN_{\mathcal G}(S)$ containing $u$.
It is clear that $cd_{\mathcal G}(S)\leq \sum_{u} cd_{\mathcal G}(S|u),$ where the summation is over all roots $u$ of $S$.

\subsection{Proof sketch}\label{subsec:sketch}
In this concise overview of the proof for Theorem~\ref{thm: main}, we outline crucial intermediate properties and emphasize the differences between the cases when $r$ is odd or even.

Consider $\mathcal G$ as a $K_{s,t}^{(r)}$-free $r$-partite $r$-uniform hypergraph with parts $V_1, \cdots , V_r$, each of size $n$, and possessing at least $n^{r-1/(s-1)-\epsilon}$ edges, where $\epsilon>0$ is a small constant.
First, we demonstrate that $\mathcal G$ can be assumed to be ``regular'' in the sense that every $(r-1)$-tuple has bounded degree.
This regularity property is proven in Lemma~\ref{lem: find regular subgraph} and simplifies the subsequent analysis.

The key ideas of the proof culminate in an auxiliary digraph $D(\mathcal G)$, where the vertex set is $\{V_1, \cdots , V_r\}$, and directed edges $V_i \rightarrow V_j$ are formed for distinct $i,j$ if there is a significant number of $s$-sets in $V_j$ rooted on vertices in $V_i$ within a relatively dense subgraph of $\mathcal G$ (refer to Definition~\ref{Def:delta-dense} for a precise description).
The main body of the proof is then divided into the following two properties, which are established in Lemmas \ref{lem: dense subgraph} and \ref{lem: 1root2root3}, respectively:
\begin{itemize}
\item[(I).] Every vertex in $D(\mathcal G)$ has non-zero in-degree, and
\item[(II).] There exist no three distinct vertices forming a directed path $V_i \rightarrow V_j\rightarrow V_k$ in $D(\mathcal G)$.
\end{itemize}
The proof of Property (II) is the most involved.
In essence, if $V_i \rightarrow V_j$ holds, it can be shown that there exist large subsets $Y\subseteq V_j$ and $Z\subseteq V_k$ for $k\notin \{i,j\}$ such that for any $y\in Y$,
the projection of $N_{\mathcal{G}}(y)$ onto $Z$ is nearly complete (see Lemma~\ref{lem: reduce to bipartite} in more details).
If $V_j\rightarrow V_k$ also holds for some $k\notin \{i,j\}$, then the number of pairs $(S,y)$ where $y\in Y$ is a root of an $s$-set $S\subseteq Z$ can be shown to be at least $(|Y||Z|^s)^{1-O(\epsilon)}$. However, due to \eqref{equ:roots}, the number of such pairs $(S,y)$ is at most $O_{r,t}(|Z|^s)$. This would lead to a contradiction and establish Property (II).

Now we can distinguish between the cases when $r$ is odd or even.
If $r$ is even, using Properties (I) and (II), one can conclude that $D(\mathcal G)$ must be isomorphic to the union of 2-cycles, say $ V_{2i-1} \leftrightarrows V_{2i}$ for $1\le i\le r/2$. This configuration is feasible, as justified in the construction in \cite{BGJS23}.
However, if $r$ is odd, Property (I) would force the existence of a directed path of length two say $V_i \rightarrow V_j\rightarrow V_k$.
This clearly contradicts Property (II) and thus completes the proof of Theorem~\ref{thm: main}.

\section{Proof of Theorem \ref{thm: main}}
In this section, we establish the proof of Theorem \ref{thm: main} by reducing it to Lemmas~\ref{lem: find regular subgraph}, \ref{lem: dense subgraph}, and \ref{lem: 1root2root3}.

Let us proceed to present the statements of these lemmas.
The first lemma demonstrates that for any $K_{s,t}^{(r)}$-free $r$-uniform hypergraph $\mathcal G$,
one can find a subgraph of $\mathcal G$ with nearly the same edge density and possessing the following useful property of being ``almost-regular''.

\begin{dfn} \label{def: regular}
Let $\mathcal G$ be a $K_{s,t}^{(r)}$-free $r$-uniform $r$-partite hypergraph with parts $V_1,\cdots , V_{r }$, each of size $n$.
Let $\varepsilon\in (0,1)$ and $\alpha>0$ be constants.
We say $\mathcal G$ is {\it $( \varepsilon,\alpha)$-regular}, if $e(\mathcal G)\ge n^{r-\frac{1}{s-1}- \varepsilon}$ and for each $i\in [r]$,  there is a constant $\Delta_i$ such that every $(r-1)$-tuple $T\in \mathcal T_i(\mathcal G)$ has bounded degree:
 $$\Delta_i/\alpha \le d_{\mathcal G}(T)\le \Delta_i, \text{ where } n^{1-\frac{1}{s-1}- \varepsilon} \le \Delta_i \le n^{1-\frac{1}{s-1}+  \varepsilon }.$$
\end{dfn}
Note that if $\varepsilon'\ge \varepsilon$, $\alpha'\ge \alpha $ and $\mathcal G$ is $( \varepsilon,\alpha)$-regular,
then $\mathcal G$ is also $( \varepsilon',\alpha')$-regular.

\begin{lem}\label{lem: find regular subgraph}
Let $\mathcal G$ be a  $K_{s,t}^{(r)}$-free $r$-uniform hypergraph on $rn$ vertices and with at least $n^{r-\frac{1}{s-1}- \varepsilon }$ edges, where $\varepsilon>0$. Then $\mathcal G $ has an $\left(\varepsilon+(\log_2n)^{-1/2},4r\log_2^r n\right)$-regular subgraph $\mathcal H$.
\end{lem}

The following definition plays a crucial role in the approach outlined in the previous section.

\begin{dfn}\label{Def:delta-dense}
Let $\mathcal G$ be a $K_{s,t}^{(r)}$-free $r$-uniform $r$-partite  hypergraph with parts $V_1,\cdots , V_{r }$, each of size $n$.
Let $\delta>0$ be a constant.
\begin{itemize}
\item Fix an $(r-1)$-tuple $T$, a vertex $u\in T$ and a vertex $v\in N_{\mathcal G}(T)$. If there are at least $ d_{\mathcal G}(T)^{s-1}/r$ many $s$-sets $S$ satisfying that $v\in S\subseteq N_{\mathcal G}(T)$ and $cd_{\mathcal G}(S|u)\ge n^{r-2-\frac{1}{s-1}-\delta}$, then we say the pair {\it $(T; v)$ is $ \delta$-dense on $u$} in $\mathcal G$.
\item Let $\mathcal H$ be a subgraph of $\mathcal G$ and $i,j\in [r]$ be two distinct integers. If for any $(r-1)$-tuple $T\in \mathcal T_j(\mathcal H)$ and any $v\in N_{\mathcal H}(T)$, $(T; v)$ is $\delta$-dense on the vertex $T\cap V_i$ in $\mathcal G$, then
we write as $V_i  \xrightarrow[\delta]{\mathcal H, \mathcal G }  V_j .$
\end{itemize}
\end{dfn}

Before turning to the statements of remaining lemmas, we would like to make several technical remarks about Definition~\ref{Def:delta-dense}.
Firstly, with appropriate choices of $\delta$ and $\varepsilon$, the condition $cd_{\mathcal G}(S|u)\ge n^{r-2-\frac{1}{s-1}-\delta}$ would imply that $u$ is a root of $S$.\footnote{This fact will be explicitly demonstrated in the proof of the first conclusion of Lemma~\ref{lem: reduce to bipartite}.}
Secondly, the notation $V_i \xrightarrow[\delta]{\mathcal H, \mathcal G} V_j$ can be equivalently expressed as follows: for any $e\in E(\mathcal H)$, the pair $(e\backslash V_j; e\cap V_j)$ is $\delta$-dense on the vertex $e\cap V_i$ in $\mathcal G$.
Lastly, if $\delta'\geq \delta$ and $V_i \xrightarrow[\delta]{\mathcal H, \mathcal G} V_j$, then we also have $V_i \xrightarrow[\delta']{\mathcal H, \mathcal G} V_j$.

\medskip

The following two lemmas will be utilized to establish Property (I) and Property (II), respectively.

\begin{lem}\label{lem: dense subgraph}
Let $\varepsilon\in (0,1)$ and $\alpha> 0$ be constants satisfying that  $\alpha=o\left(n^{\varepsilon/s}\right)$.
Suppose $\mathcal G$ is an  $(\varepsilon,\alpha)$-regular $K_{s,t}^{(r)}$-free $r$-uniform $r$-partite  hypergraph with parts $V_1,\cdots , V_{r }$, each of size $n$. For any part $V_j$,
there exists an $(\varepsilon+\log_n 4r,4r^2\alpha)$-regular subgraph $\mathcal H\subseteq \mathcal G$ and a distinct part $V_i$ such that
$V_i \xrightarrow[\delta]{\mathcal  H, \mathcal G } V_j$, where $\delta:=(s+1)\varepsilon$.
\end{lem}

\begin{lem}\label{lem: 1root2root3}
Let $\varepsilon, \delta, \alpha>0$ satisfy $6(s+1)(\varepsilon + \delta)\le 1$ and $\alpha=o\left(n^{\varepsilon}\right)$.
Let $n$ be sufficiently large and $\mathcal H_1\subseteq\mathcal H\subseteq \mathcal G_1 \subseteq\mathcal G$ be a sequence of $(\varepsilon, \alpha) $-regular $K_{s,t}^{(r)}$-free $r$-uniform $r$-partite hypergraphs  with parts $V_1,\cdots , V_{r }$, each of size $n$. If
$V_i \xrightarrow[\delta]{\mathcal  H_1, \mathcal H}  V_j \xrightarrow[\delta]{\mathcal  G_1,\mathcal G }  V_k$ holds for $j\notin \{i,k\}$,
then $k=i$.
\end{lem}

Finally, we are prepared to prove Theorem \ref{thm: main}, assuming Lemmas~\ref{lem: find regular subgraph}, \ref{lem: dense subgraph}, and \ref{lem: 1root2root3}.

\begin{proof}[\bf Proof of Theorem \ref{thm: main}.]
Let $\varepsilon= \frac{1}{6(s+2)^2}$ throughout this proof. It suffices to show that for any odd $r\geq 3$ and sufficiently large integer $n$,
every $K_{s,t}^{(r)}$-free $r$-uniform hypergraph $\mathcal G$ on $rn$ vertices has at most $n^{r-\frac{1}{s-1}-\varepsilon }$ edges.
Suppose for a contradiction that there exists such an $r$-uniform hypergraph $\mathcal G$ with more than $n^{r-\frac{1}{s-1}-\varepsilon }$ edges.

By Lemma~\ref{lem: find regular subgraph}, $\mathcal G$ contains an $\left(\varepsilon_1,\alpha_1\right)$-regular subgraph $\mathcal G_1$, where $\varepsilon_1=\varepsilon +(\log_2n)^{-1/2}$ and $\alpha_1=4r\log_2^r n$.
Then $\mathcal G_1$ is balanced $r$-partite, say with parts $V_1,\cdots, V_r$, each of size $n$.
Let $$\varepsilon_i=\varepsilon_1+(i-1)\log_{n}4r \mbox{ ~and~ } \alpha_i=(4r^2)^{i-1}\alpha_1 \mbox{ for all } i\geq 1.$$
Also let $$\varepsilon^*=\varepsilon+2r(\log_2 n)^{-1/2} \mbox{ ~and~ } \alpha^*=(4r^2)^{r+2}\log_2^r n.$$
We note that for each $1\leq i\leq r+2$,
\begin{equation}\label{equ:alpha*}
\mbox{$\varepsilon^*\geq \varepsilon_i$, ~$\alpha^*\geq \alpha_i$, ~and $\alpha_i$ satisfies the condition of Lemma~\ref{lem: dense subgraph}.}
\end{equation}

We will iteratively apply Lemma~\ref{lem: dense subgraph} to obtain a sequence of
$K_{s,t}^{(r)}$-free $r$-uniform $r$-partite hypergraphs $\mathcal G_1\supseteq \mathcal G_2 \supseteq \cdots \supseteq \mathcal G_{r+2}$ on $rn$ vertices as follows.
Initially, let $a_1=1$; applying Lemma~\ref{lem: dense subgraph} with respect to $\mathcal G_1$ (which is $\left(\varepsilon_1,\alpha_1\right)$-regular) and the part $V_{a_1}$, there exist an $(\varepsilon_2,\alpha_2)$-regular subgraph $\mathcal G_2\subseteq \mathcal G_1$ and an index $b_1\neq a_1$ such that
$V_{b_1} \xrightarrow[(s+1)\varepsilon_1]{\mathcal G_2, \mathcal G_1 } V_{a_1}$;
then applying Lemma~\ref{lem: dense subgraph} with respect to $\mathcal G_2$ and the part $V_{b_1}$,
there exist an $(\varepsilon_3,\alpha_3)$-regular subgraph $\mathcal G_3\subseteq \mathcal G_2$ and an index $c_1\neq b_1$ such that
$V_{c_1} \xrightarrow[(s+1)\varepsilon_2]{\mathcal G_3, \mathcal G_2 } V_{b_1}$.
Now assume that the sequence has been defined for $\mathcal G_1\supseteq \cdots \supseteq \mathcal G_{2i-1}$ for some $2\leq i\leq (r+1)/2$.
We choose an index $a_i\in [r]\backslash \left(\cup_{1\leq j\leq i-1} \left\{a_j,b_j,c_j\right\}\right)$,\footnote{We will see later that such an index is always valid as long as $i\leq (r+1)/2$.}
and then apply Lemma~\ref{lem: dense subgraph} twice to get subgraphs $\mathcal G_{2i+1}\subseteq \mathcal G_{2i}\subseteq \mathcal G_{2i-1}$ and indices $b_i, c_i\in [r]$ such that
$$V_{c_i} \xrightarrow[(s+1)\varepsilon_{2i}]{\mathcal G_{2i+1}, \mathcal G_{2i} } V_{b_i} \xrightarrow[(s+1)\varepsilon_{2i-1}]{\mathcal G_{2i}, \mathcal G_{2i-1} } V_{a_i}, \mbox{where $\mathcal G_{j}$ is $(\varepsilon_{j},\alpha_{j})$-regular for } j\in \{2i, 2i+1\}.$$
Let $\delta^*=(s+1)\varepsilon^*$.
Then as $n$ is sufficiently large, it follows that
\begin{equation}\label{equ:eps+del}
6(s+1)(\varepsilon^* + \delta^*)=6(s+1)(s+2)\left(\frac{1}{6(s+2)^2}+2r(\log_2 n)^{-1/2}\right)<1.
\end{equation}
In view of \eqref{equ:alpha*} and the remarks after Definitions~\ref{def: regular} and \ref{Def:delta-dense},
we see that $\mathcal G_\ell$ is $(\varepsilon^*,\alpha^*)$-regular for each $1\leq \ell \leq r+2$,
and
\begin{equation}\label{equ:ci->ai}
V_{c_i} \xrightarrow[\delta^*]{\mathcal G_{2i+1}, \mathcal G_{2i} } V_{b_i} \xrightarrow[\delta^*]{\mathcal G_{2i}, \mathcal G_{2i-1} } V_{a_i} \mbox{ holds for each } 1\leq i\leq (r+1)/2.
\end{equation}

By \eqref{equ:eps+del} and the fact $\alpha^*=o\left(n^{\varepsilon}\right)$, using Lemma~\ref{lem: 1root2root3},
we can easily conclude that $c_i=a_i$ for all $1\leq i\leq (r+1)/2$.
By the choice of $a_i$, evidently $a_i$ is distinct from the indices in $\cup_{1\leq j\leq i-1} \left\{a_j,b_j\right\}$.
We claim that $b_i$ is also distinct from the indices in $\cup_{1\leq j\leq i-1} \left\{a_j,b_j\right\}$.
Otherwise, $b_i\in \{a_j,b_j\}$ for some $1\leq j\leq i-1$, which, combining \eqref{equ:ci->ai} for the index $j$ (also using $c_j=a_j$),
would yield that
$$\mbox{either } V_{a_i} \xrightarrow[\delta^*]{\mathcal G_{2i+1}, \mathcal G_{2i} } V_{b_i}\xrightarrow[\delta^*]{\mathcal G_{2j+1}, \mathcal G_{2j} } V_{b_j} \mbox{ (if } b_i=a_j), \mbox{ or } V_{a_i} \xrightarrow[\delta^*]{\mathcal G_{2i+1}, \mathcal G_{2i} } V_{b_i}\xrightarrow[\delta^*]{\mathcal G_{2j}, \mathcal G_{2j-1} } V_{a_j}\mbox{ (if } b_i=b_j).$$
Using Lemma~\ref{lem: 1root2root3} again, we then deduce that either $a_i=b_j$ or $a_i=a_j$, a contradiction to the choice of $a_i$, proving the claim.
This shows that $\{a_i,b_i\}$ is disjoint from $\{a_j,b_j\}$ whenever $1\leq i\neq j\leq (r+1)/2$.
Consequently, $|\bigcup_{1\leq i\leq (r+1)/2} \{a_j,b_j\}|=2\cdot (r+1)/2=r+1$.
However, this contradicts the fact that $\bigcup_{1\leq i\leq (r+1)/2} \{a_j,b_j\}\subseteq [r]$.
This final contradiction completes the proof of Theorem~\ref{thm: main}.
\end{proof}

\section{Proof of lemmas}
This section is devoted to the proofs of Lemmas~\ref{lem: find regular subgraph},  \ref{lem: dense subgraph} and \ref{lem: 1root2root3}.

\subsection{Finding $(\varepsilon,\alpha)$-regular subgraphs}
In this subsection, we establish Lemma~\ref{lem: find regular subgraph} along with several related properties regarding $(\varepsilon,\alpha)$-regularity.
The first lemma will be frequently used later to provide upper bounds on the (co-)degrees of subsets.

\begin{lem}\label{lem: set degree}
Let $\mathcal G$ be a $K_{s,t}^{(r)}$-free $r$-uniform balanced $r$-partite hypergraph on $rn$ vertices.
Suppose there is a constant $\Delta$ such that $d_{\mathcal G}(T)\le \Delta$ holds for every $(r-1)$-tuple $T$.
Let $A$ be a $k$-tuple and $S$ be an $s$-set of $\mathcal{G}$. Then the following hold that
$$ d_{\mathcal G}(A)\le  \Delta n^{r-k-1  }  \text{ ~ and ~ } cd_{\mathcal G}(S)\le rt\Delta n^{r-3}   .$$
Moreover, if $\mathcal G$ is $( \varepsilon,\alpha)$-regular, then
$$ d_{\mathcal G}(A)\le n^{r-k -\frac{1}{s-1}+  \varepsilon}  \text{ ~ and ~ }  cd_{\mathcal G}(S)\le  rtn^{r-2 -\frac{1}{s-1}+  \varepsilon  }  .$$
\end{lem}
\begin{proof}
By definition, $N_{\mathcal G}(A)$ is an $(r-k)$-uniform $(r-k)$-paritite hypergraph with each part of size $n$.
On average, there is an $(r-k-1)$-set $B$ such that at least $d_{\mathcal G}(A) /n^{r-k-1} $ edges in $N_{\mathcal G}(A)$ containing $B$.
It also means that $\mathcal G$ has at least $d_{\mathcal G}(A) /n^{r-k-1} $ edges containing the $(r-1)$-set  $A\cup B$.
Consequently, one can get $\Delta \ge d_{\mathcal G}(A\cup B)\ge d_{\mathcal G}(A) /n^{r-k-1}$.

Let $R$ be the set of roots of the $s$-set $S$. Since $\mathcal G$ is $K_{s,t}^{(r)}$-free, by \eqref{equ:roots} we have $|R|<rt$.
Fix a vertex $v\in S$. For any $u\in R$, by the previous paragraph we have $d_{\mathcal G}(\{u,v\})\le \Delta n^{r-3}.$
Thus we can derive that $cd_{\mathcal G}(S)\le \sum_{u\in R}d_{\mathcal G}(\{u,v\})\le rt\Delta n^{r-3}.$

If $\mathcal G$ is $( \varepsilon,\alpha)$-regular, then we can substitute $\Delta $ with $n^{1-\frac{1}{s-1}+ \varepsilon }$ to get the desired inequalities.
\end{proof}

We now present the proof of Lemma~\ref{lem: find regular subgraph} using standard deletion arguments.

\begin{proof}[\bf Proof of Lemma~\ref{lem: find regular subgraph}]
Let $\mathcal G$ be a $K_{s,t}^{(r)}$-free $r$-uniform hypergraph on $rn$ vertices and with at least $n^{r-\frac{1}{s-1}- \varepsilon}$ edges.
It is well known that $\mathcal G$ has an $r$-partite subgraph $\mathcal G'$ with parts $V_1,V_2,\cdots, V_r$ of size $n$
such that $e(\mathcal G')\ge \frac{r!}{r^r}e(\mathcal G)\geq \frac{r!}{r^r}n^{r-\frac{1}{s-1}- \varepsilon}$.

We first consider $(r-1)$-tuples in $\mathcal T_1(\mathcal G')$ and partition them into sub-families based on their degrees.
We define $\mathcal T_{1,j}\subseteq \mathcal T_1(\mathcal G')$ as the set consisting of all $T\in \mathcal T_1(\mathcal G')$ with $2^{j-1}\le d_{\mathcal G'}(T)\le 2^{j}$ for $j\in [\log_2n]$.
By averaging, we can find a set $\mathcal T_{1,k}$ such that the number of the corresponding edges of $\mathcal G'$ is at least $e(\mathcal G')/\log_2n$.
Let $\mathcal G_1$ be the union of these edges, and let $\Delta_1 = 2^k$.
Then we have $e(\mathcal G_1)\ge e(\mathcal G')/\log_2n$ and $\Delta_1/2\le d_{\mathcal G_1}(T)\le \Delta_1$ for any $T\in \mathcal{T}_1(\mathcal{G}_1)$.
Consequently, we obtain $|\mathcal T_1(\mathcal G_1)|\le e(\mathcal G_1)/(\Delta_1/2)\le   2e(\mathcal G')/ \Delta_1 $.
Similarly, we can get $\mathcal G_2\subseteq \mathcal G_1$ and a constant $\Delta_2$ such that $e(\mathcal G_2)\ge e(\mathcal G_1)/\log_2n$ and $\Delta_2/2\le d_{\mathcal G_2}(T)\le \Delta_2$ for every $T\in \mathcal T_2(\mathcal G_2)$.
Repeating this process, we can obtain subgraphs $\mathcal G_i\subseteq \mathcal G_{i-1}$ and constants $\Delta_i$ for all $2\leq i\leq r$
such that $e(\mathcal G_i)\ge e(\mathcal G_{i-1})/\log_2n$ and $\Delta_i/2\le d_{\mathcal G_i}(T)\le \Delta_i$ for every $T\in \mathcal T_i(\mathcal G_i)$.
Then $|\mathcal T_i(\mathcal G_i)|\le e(\mathcal G_i)/(\Delta_i/2)\le 2e(\mathcal G')/ \Delta_i$ for any $i\in [r]$.

Collecting the above estimations in $\mathcal G_r$, we have $e(\mathcal G_r)\ge e(\mathcal G')/\log_2^rn$,
$|\mathcal T_i(\mathcal G_r)|\le |\mathcal T_i(\mathcal G_i)|\leq 2e(\mathcal G')/ \Delta_i$ for all $i\in [r]$,
and for each $T\in \mathcal T_i(\mathcal G_r)$ (which is also in $\mathcal T_i(\mathcal G_i)$),
$d_{\mathcal G_r}(T)\le d_{\mathcal G_i}(T)\leq \Delta_i$.

Next, we want to identify a subgraph $ \mathcal H\subseteq \mathcal G_r$ with a proper lower bound on $d_{\mathcal H}(T)$ for any $(r-1)$-tuple $T$ through the following deleting process.
Initially, let $ \mathcal H= \mathcal G_r$.
If there is an $(r-1)$-tuple $T\in \mathcal T_i(H)$ for some $i\in [r]$ with $d_{\mathcal H}(T)<\Delta_i/ 4r\log_2^rn$, then we delete all edges containing $T$, and denote the resulting hypergraph as $ \mathcal H$.
Repeat this process until either $ \mathcal H$ is empty or every $T\in \mathcal T_i(H)$ satisfies $d_{\mathcal H}(T)\geq \Delta_i/ 4r\log_2^rn$.
The number of edges we deleted is
$$e(\mathcal G_r)-e(\mathcal H)\le \sum_{ i\in [r]} | \mathcal T_i(\mathcal G_r)|\cdot \frac{\Delta_i}{4r\log_2^rn}
\le \sum_{ i\in [r]}  \frac{2e(\mathcal G')}{\Delta_i}\cdot \frac{\Delta_i}{4r\log_2^rn}=\frac{e(\mathcal G')}{2\log_2^rn}\le \frac{e(\mathcal G_r)}{2}.$$
Therefore, $\mathcal H$ is a non-empty subgraph of $\mathcal G_r\subseteq \mathcal G$ with
$$e(\mathcal H)\ge \frac{e(\mathcal G_r)}{2}\ge \frac{e(\mathcal G')}{2\log_2^rn}\ge \frac{n^{r-\frac{1}{s-1}- \varepsilon}}{(r\log_2n)^r}\ge n^{r-\frac{1}{s-1}- \varepsilon'}, \mbox{ where }\varepsilon'=\varepsilon+(\log_2 n)^{-1/2},$$
and $\Delta_i/4r\log_2^rn \le d_{\mathcal H}(T)\leq d_{\mathcal G_r}(T)\le \Delta_i$ for every $T\in \mathcal T_i(H)$.

To show $\mathcal H$ is the desired $\left(\varepsilon',4r\log_2^r n\right)$-regular subgraph of $\mathcal G$,
it remains to bound every $\Delta_i$ for $i\in [r]$.
For any $i\in [r]$, since $\Delta_i$ is the upper bound of $d_{\mathcal H}(T)$ for every $T\in \mathcal T_i(H)$,
it is clear that $$\Delta_i\ge e(\mathcal H)/n^{r-1}\ge n^{1-\frac{1}{s-1}- \varepsilon'}.$$
So it will suffice to show $\Delta_i\leq n^{1-\frac{1}{s-1}+\varepsilon'}$ for every $i\in [r]$.
By symmetry, we may assume that $\Delta_1$ is the maximum one among all $\Delta_i$.
Let us count the number $m$ of pairs $(S, T)$,
where $T\in \mathcal T_1(\mathcal H)$  and $S$ is an $s$-set in $N_{\mathcal H}(T)$.
Since $\sum_{T\in \mathcal  T_1(\mathcal H) }d_{\mathcal H}(T)= e(\mathcal H)$ and for each such $T$,
$d_{\mathcal H}(T)\ge \Delta_1/ 4r\log_2^{r }n\gg s$,
we can obtain that
$$m=  \sum_{T\in \mathcal  T^1(\mathcal H)} \binom{d_{\mathcal H}(T)}{s}
\ge e(\mathcal H) \cdot \left(\frac{d_{\mathcal H}(T)}{s}\right)^{s-1}
\ge \frac{n^{r-\frac{1}{s-1}- \varepsilon}}{(r\log_2n)^r}\cdot \left(\frac{\Delta_1}{4rs\log_2^{r }n } \right)^{s-1}.$$
By averaging, there exists an $s$-set $S$ which is contained in $N_{\mathcal H}(T)$ for at least $m/n^{s}$ many $(r-1)$-tuples $T$.
Using Lemma~\ref{lem: set degree}, we have
$$rt\Delta_1 n^{ r-3}\ge  cd_{\mathcal H}(S) \ge \frac{m}{n^s }
\ge \frac{n^{r-s-\frac{1}{s-1}-\varepsilon}\Delta_1^{s-1}}{r^r\cdot (4rs\log_2^{r}n)^{s}}.$$
As $n$ is sufficiently large,
this gives $\Delta_1^{s-2}\le n^{s-3+\frac{1}{s-1}+\varepsilon}\log_2^{rs +1 }n$.
As $s\geq 3$, it further implies that
$$\max_{i\in [r]} \Delta_i=\Delta_1\le n^{1 -\frac{1}{s-1}+ \frac{\varepsilon}{s-2}}\left(\log_2 n\right)^{\frac{rs+1}{s-2}}\le n^{1 -\frac{1}{s-1}+ \varepsilon'}.$$
Here, $\varepsilon'= \varepsilon + (\log_2n)^{-1/2}$.
Hence $\mathcal G$ has an $(\varepsilon',4r\log_2^r n)$-regular subgraph $\mathcal H$.
\end{proof}

We would like to point out that the proof of Lemma~\ref{lem: find regular subgraph} can be slightly modified to show that for any $r\geq 3$, $s,t\geq 2$ and sufficiently large $n$,
$\ex(n,K^{(r)}_{s,t})\leq n^{r-\frac{1}{s-1}}\log^{2r} n$ holds.

Applying a similar deletion argument given in the above proof, we can also derive the following.

\begin{lem} \label{coro: regular graph subgraph}
Let $\mathcal G$  be   an $(\varepsilon,\alpha)$-regular $K_{s,t}^{(r)}$-free $r$-uniform balanced $r$-partite hypergraph on $rn$ vertices.
Let $c>0$ be a constant.
If $\mathcal G'$ is a subgraph of $\mathcal G$ with at least $e(\mathcal G)/c$ edges,
then $\mathcal G'$ has an $(\varepsilon +\log_n{2c}, 2\alpha cr)$-regular subgraph $\mathcal H$.
\end{lem}
\begin{proof}
We apply the deletion argument to get a subgraph $\mathcal H\subseteq \mathcal G'$ as follows.
Initially set $\mathcal H= \mathcal G'$.
If $\mathcal H$ has an $(r-1)$-tuple $T$ with $d_{\mathcal H}(T)<d_{\mathcal G}(T)/2cr$,
then we delete all edges containing $T$, and still denote the resulting hypergraph as $\mathcal H$.
Repeat this process until $ \mathcal H$ is empty or every $(r-1)$-tuple $T$ of $\mathcal H$ satisfies that $d_{\mathcal H}(T)\geq d_{\mathcal G}(T)/2cr$.
The number of edges we deleted is
$$e(\mathcal G')-e(\mathcal H)\le \sum_{ T\in \mathcal T(\mathcal G) } \frac{ d_{\mathcal G}(T)}{2cr} = \frac{r\cdot e(\mathcal G)}{2cr} = \frac{e(\mathcal G')}{2}.$$
So we have $e(\mathcal H)\ge \frac{e(\mathcal G')}{2}\ge \frac{e(\mathcal G )}{2c} \ge n^{r-\frac{1}{s-1}-   \varepsilon - \log_n{2c}}$  and $d_{\mathcal G}(T)/2cr \le  d_{\mathcal H}(T) \le d_{\mathcal G}(T)$ for each $T\in \mathcal T (\mathcal H)$.
Since $\mathcal G$ is $(\varepsilon,\alpha)$-regular,
there exists $\Delta_i\in [n^{1-\frac{1}{s-1}-\varepsilon},n^{1-\frac{1}{s-1}+\varepsilon}]$
such that for each $T\in \mathcal T_i(\mathcal G)$, $\Delta_i/\alpha \le d_{\mathcal G}(T)\le \Delta_i$.
This implies that for each $T\in \mathcal T_i (\mathcal H)$, we have
$\Delta_i/2\alpha cr \le d_{\mathcal G}(T)/2cr \le  d_{\mathcal H}(T) \le d_{\mathcal G}(T)\le \Delta_i$.
Therefore, $\mathcal H$ is an $(\varepsilon +\log_n{2c},  2\alpha cr)$-regular subgraph of $\mathcal G'$.
\end{proof}

\subsection{Finding $\delta$-dense structures: Property (I)}
In this subsection, we prove Lemma~\ref{lem: dense subgraph}.

\begin{proof}[\bf Proof of Lemma~\ref{lem: dense subgraph}]
Let $\varepsilon\in (0,1)$ and $\alpha>0$ satisfy that $\alpha=o\left(n^{\varepsilon/s}\right)$.
Let $\mathcal G$ be an $(\varepsilon,\alpha)$-regular $K_{s,t}^{(r)}$-free $r$-uniform $r$-partite hypergraph with parts $V_1,\cdots , V_{r }$, each of size $n$.
Without loss of generality, we may assume that $j=r$.
Our goal is to show that there exists an $(\varepsilon+\log_n4r,4r^2\alpha)$-regular subgraph $\mathcal H\subseteq \mathcal G$ and an integer $i\in [r]\backslash \{r\}$ such that
$ V_i \xrightarrow[(s+1)\varepsilon]{\mathcal H, \mathcal G } V_r.$

Let us first introduce some notation needed in this proof.
For an $s$-set $S$ and an $(r-1)$-tuple $T$ in $\mathcal G$ with $S\subseteq N_{\mathcal G}(T)$,
we define the {\it co-degree of $S$ under $T$} as follows:
$$ cd_{\mathcal G}(S | T )=\max  \{ cd_{\mathcal G}(S|u): u \in T \text{ and }  u \text{ is a root of } S \}.\footnote{Note that this is well-defined as there exists at least one vertex $u\in T$ which is a root of $S$.}$$
We say the pair $(S,T)$ is {\it small} if $cd_{\mathcal G}(S |T )< n^{r-2-\frac{1}{s-1}- (s+1)\varepsilon}$ and {\it large} otherwise.

We define an auxiliary function $f: E(\mathcal G)\to \{0,1,\cdots, r-1\}$ on the edges of $\mathcal G$ as follows.
For an edge $e=\{v_1,\cdots, v_r\}\in E(\mathcal G)$ with $v_i\in V_i$,
let $T_e=\{v_1, \cdots ,v_{r-1}\}$ and let $\mathcal{S}(e)$ be the family consisting of all $s$-sets $S$ satisfying that $v_r\in S \subseteq N_{\mathcal G}(T_e)$.
We define $f(e)=0$ if there are at least $(n^{1-\frac{1}{s-1}- \varepsilon}/(\alpha\log_2 n) )^{s-1}= n^{s-2-  (s-1)\varepsilon}/(\alpha\log_2 n)^ { s-1}$  many  $s$-sets $S\in \mathcal S(e)$ such that the pair $(S,T_e)$ is small.
Subsequently, if $f(e)\neq 0$, considering that $d_{\mathcal G}(T_e)\geq n^{1-\frac{1}{s-1}- \varepsilon}/\alpha$ ,
there are $(1-o(1))d_{\mathcal G}(T_e)^{s-1}$ many $s$-sets $S\in \mathcal{S}(e)$ such that $(S,T_e)$ is large.
Note that in this case, for every $s$-set $S\in \mathcal{S}(e)$ with large $(S,T_e)$, there exists a root of $S$, denoted as $v_k\in T_e$, satisfying $cd_{\mathcal G}(S|v_k)\geq n^{r-2-\frac{1}{s-1}- (s+1)\varepsilon}$.
We will refer to such $S\in \mathcal{S}(e)$ as having {\it index $k$}, where $k\in \{1,\cdots, r-1\}$.\footnote{If there are multiple choices of $k$, we arbitrarily select one of them.}
Let $\ell$ be the index such that the number of $s$-sets $S\in \mathcal{S}(e)$ with index $\ell$ is maximum among all indices in $\{1,\cdots, r-1\}$.
By averaging, this number is at least $(1-o(1))d_{\mathcal G}(T_e)^{s-1}/(r-1)\ge d_{\mathcal G}(T_e)^{s-1}/r$.
In this case, we define $f(e)=\ell$ and according to Definition~\ref{Def:delta-dense}, $(T_e; v_r) $ is $(s+1)\varepsilon$-dense on $v_\ell$ in $\mathcal G$.

Let $m$ be the number of edges $e$ in $\mathcal G$ with $f(e)=0$.
Now we demonstrate that to complete this proof, it is sufficient to show that $m=o(e(\mathcal G))$.
Suppose indeed $m=o(e(\mathcal G))$.
Then by averaging, there exists an integer $i\in \{1,\cdots, r-1\}$ such that there are at least $(1-o(1))e(\mathcal G)/(r-1)\ge e(\mathcal G)/r $ many edges $e\in E(\mathcal G)$ with $f(e)=i$.
By applying Lemma~\ref{coro: regular graph subgraph} (with the constant $c=r$),
one can get an $(\varepsilon+\log_n 4r,4r^2\alpha)$-regular subgraph $\mathcal H$ from these edges.
Every edge $e$ in $\mathcal H$ has $f(e)=i$, which also means that $(e\backslash V_r; e\cap V_r) $ is $(s+1)\varepsilon$-dense on the vertex $e\cap  V_i$ in $\mathcal G$. Therefore, $V_i \xrightarrow[(s+1)\varepsilon]{\mathcal{H},\mathcal{G}} V_r$ holds
and $\mathcal H$ is the desired subgraph of $\mathcal G$.

It remains to show that $m=o(e(\mathcal G))$.
We count the number of pairs $(S, e)$, where $S$ is an $s$-set in $V_r$ and $e\in E(\mathcal G)$ such that $(S,T_e)$ is small.
By definition of $f(e)=0$, we see that this number, denoted by $M$, is at least $m\cdot n^{s-2- (s-1)\varepsilon}/(\alpha\log_2 n)^ { s-1}$.
Now fix an $s$-set $S_0$.
Recall the definition of smallness.
Any $(r-1)$-tuple $T$ such that the pair $(S_0,T)$ is small must contain a root $u$ of $S$ with
$cd_{\mathcal G}(S|u)< n^{r-2-\frac{1}{s-1}- (s+1)\varepsilon}$; call such root $u$ {\it small}.
Thus by \eqref{equ:roots}, the number of such $(r-1)$-tuples $T$ is at most
$\sum_{u} cd_{\mathcal G}(S|u)< rtn^{r-2-\frac{1}{s-1}- (s+1)\varepsilon}$,
where the summation is over all small roots $u$ of $S$.
There are at most $n^s$ choices of $s$-sets $S_0$,
and each small pair $(S_0,T )$ can contribute $s$ pairs $(S_0,e)$ with $T\subseteq e\subseteq S_0\cup T$ to the counting $M$.
Thus we have
$$m\cdot n^{s-2- (s-1)\varepsilon}/(\alpha\log_2 n)^ { s-1}\leq M < sn^s\cdot rt   n^{r-2-\frac{1}{s-1}- (s+1)\varepsilon}.$$
Note that $(\alpha\log_2 n)^{s-1}= o\left( (n^{\varepsilon/s} \log_2 n) ^{s-1}\right)=o\left(n^{\varepsilon }  \right)$  and $e(\mathcal G)\ge n^{r-\frac{1}{s-1}- \varepsilon}$.
So the above inequality implies that $m\le srt\cdot (\alpha\log_2 n)^ { s-1}\cdot n^{r -\frac{1}{s-1}- 2\varepsilon}= o(e(\mathcal G))$.
We have proved Lemma~\ref{lem: dense subgraph}.
\end{proof}

\subsection{Finding $\delta$-dense structures: Property (II)}
In this subsection, we prove Lemma~\ref{lem: 1root2root3}.
Before presenting the proof, we need to establish two technical lemmas.
The first one involves some averaging statements for bipartite graphs.

\begin{lem}\label{lem: bipartite average}
Let $G=(A,B)$ be a bipartite graph with $e(G)\geq \rho|A||B|$ for some $\rho\in (0,1)$.
\begin{itemize}
  \item[(1).] There are at least $\rho |A|/2$ vertices $a\in A$ with $|N_G(a)\cap B |\ge \rho |B |/2$.
  \item[(2).] Let $s$ be a positive integer. If $\rho |A|\gg s$, then there are at least $(\rho |A|)^s/(3s!)$  many $s$-sets in $A$ that have at least $\rho^s B/3$ common neighbors in $G$.
\end{itemize}
\end{lem}
\begin{proof}
For (1), let $A'$ be the set of vertices $a\subseteq A$ with $|N(a)\cap B |\ge \rho |B|/2$.
Then we have $\rho|A||B|\leq e(G)\leq |A'||B| + (|A|-|A'|)\rho |B|/2$.
This gives that $\rho|A||B |/2\le (1-\rho/2)|A'||B|$ and thus $|A'|\ge \rho |A|/(2-\rho)\ge \rho |A|/2$, as desired.

For (2), we construct an auxiliary bipartite graph $H$ from $G$ as follows.
The two parts of $H$ are $B$ and $\mathcal C$, where $\mathcal C$ denotes the family of all $s$-sets in $A$.
Let $bc\in E(H)$ if and only if $b\in B$ and $c\in \mathcal C$ form a $K_{1,s}$ in $G$.
%So $e(H)$ is the number of $K_{s,1}$ in $G$, with $s$ vertices in $A$ part.
In view of $e(G)= \rho|A||B|$ and $\rho |A|\gg s$, by Jensen's inequality, we have
$$e(H)  =\sum_{b\in B}\binom{d_G(b)}{s}\ge |B|\binom{ \rho|A|}{s} =  (1-o(1)) \rho^s|\mathcal C||B|, \mbox{ where } |\mathcal C|=\binom{|A|}{s}.$$
By applying the conclusion (1) for $H$, there are at least $(1-o(1)) \rho^s|\mathcal C|/2\geq (\rho |A|)^s/(3s!)$  many  vertices $c\in \mathcal C$ with $|N_H(c)\cap B |\ge (1-o(1))\rho^s  |B |/2 \ge \rho^s |B |/3$. This completes the proof.
\end{proof}

The following lemma provides the crucial techniques for proving Lemma~\ref{lem: 1root2root3}.
Roughly speaking, given the assumption $V_1 \xrightarrow[\delta]{\mathcal H, \mathcal G} V_2$,
it reveals some dense structures concerning the ``adjacency'' between $V_2$ and $V_1$, as well as between $V_2$ and any predetermined part $V_j$.

\begin{lem}\label{lem: reduce to bipartite}
Let $\varepsilon, \delta, \alpha>0$ satisfy $6(s+1)(\varepsilon + \delta)\le 1$ and $\alpha=o\left(n^{\varepsilon}\right)$.
Let $\mathcal G$ be an $(\varepsilon,\alpha)$-regular $K_{s,t}^{(r)}$-free $r$-uniform balanced $r$-partite hypergraph on $rn$ vertices with parts $V_1,\cdots, V_r$.

Fix $T=\{v_1,v_3,\cdots, v_r\}\in \mathcal T_2(\mathcal G)$, where $v_i\in V_i$ for $i\in [r]\backslash \{2\}$.
Let $X$ be a subset of $N_{\mathcal G}(T)\subseteq V_2$ such that for every vertex $v\in X$, $(T;v)$ is $\delta $-dense on $v_1$ in $\mathcal G$.
Then the following hold.
\begin{itemize}
\item[(1).] If $|X|\gg n^{\varepsilon+\delta}$, then $X$ contains at least $n^{-s\varepsilon-s\delta }|X|^s/(3s!r^s)$ different $s$-sets rooted on $v_{1}$.
\item[(2).] Suppose $X\neq \emptyset$. Then for any given $j\in [r]\backslash \{1,2\}$,
there exist subsets $Y\subseteq N_{\mathcal G}(T)$, $Z\subseteq V_j$, and an $(r-3)$-tuple $R\subseteq V(\mathcal G)\backslash (V_1\cup V_2\cup V_j)$ such that $|Y|\ge n^{1-\frac{1}{s-1}- 2 \varepsilon -\delta}/2r\alpha$, $n^{1-\frac{1}{s-1}-\delta} \le |Z| \le n^{1-\frac{1}{s-1} +\varepsilon}$, and for any $y \in Y$, $|N_{\mathcal G}(\{v_1, y\}\cup R)\cap Z|\ge n^{1-\frac{1}{s-1} - \varepsilon-2\delta}/2$.
\end{itemize}
\end{lem}
\begin{proof}
We fix such an $(r-1)$-tuple $T$.
For each $v\in N_{\mathcal G}(T)$, let $\mathcal A(v)$ be the set of $s$-sets $S$ such that $v\in S\subseteq N_{\mathcal G}(T)$ and $cd_{\mathcal G}(S|v_{1})\ge n^{r-2-\frac{1}{s-1}- \delta}$.
If $v\in X$, then $(T;v)$ is $\delta $-dense on $v_1$ in $\mathcal G$ and thus $|\mathcal A(v)|\ge d_{\mathcal G}(T)^{s-1}/ r$.
Let $A(v)=\bigcup_{S\in \mathcal A(v)}S$.
Then $A(v)\subseteq N_{\mathcal G}(T)$, and as clearly $\binom{|A(v)|}{s-1}\geq |\mathcal A(v)|$, we have $|A(v)|\ge |\mathcal A(v)|^{1/(s-1)} \ge d_{\mathcal G}(T)/r$.
Note that for each $u\in A(v)$,
there exists some $S\in \mathcal A(v)$ with $\{u,v\}\subseteq S$.
This implies that for each $u\in A(v)$,
\begin{equation}\label{equ:d(u,v1)}
|N_{\mathcal G}(\{u,v_{1}\})\cap N_{\mathcal G}(\{v,v_{1}\})|\ge cd_{\mathcal G}(S|v_{1})\ge n^{r-2-\frac{1}{s-1} -\delta}.
\end{equation}

We first consider the conclusion (1).
We have seen that for all $v\in X$, $A(v)\subseteq N_{\mathcal G}(T)$ has size at least $d_{\mathcal G}(T)/r$.
By averaging, there exists $u_0\in N_{\mathcal G}(T)$ and a subset $X'\subseteq X$ with $|X'|\ge |X|/r$ such that $u_0\in A(v)$ for every $v\in X'$.
Let $\mathcal B= N_{\mathcal G}(\{u_0,v_{1}\})$ be a subset of $(r-2)$-tuples.
Then we have $$n^{r-2-\frac{1}{s-1} -\delta}\le |\mathcal B|\le n^{r-2-\frac{1}{s-1}+ \varepsilon},$$
where the first inequality holds by \eqref{equ:d(u,v1)} and the second inequality is given by Lemma~\ref{lem: set degree}.
We now define a bipartite graph $H=(X',\mathcal B)$ as follows. For $v\in X'$ and $B\in \mathcal B$,
we define $vB\in E(H)$ if and only if $\{v,v_{1}\}\cup B\in E(\mathcal G)$.
Note that it means $N_H(v)=\mathcal B\cap N_\mathcal{G}(\{v,v_1\})$ for $v\in X'$.
By \eqref{equ:d(u,v1)}, we see that each $v\in X'$ has degree at least $n^{r-2-\frac{1}{s-1} -\delta}$ in $H$.
Consequently, $e(H)\ge |X'|n^{r-2-\frac{1}{s-1} -\delta}\ge n^{ -  \varepsilon -\delta}|X'||\mathcal B|$.
Since $|X'|\ge |X|/r\gg n^{\varepsilon +\delta}$, applying Lemma~\ref{lem: bipartite average} (2),
there are at least $(n^{-\varepsilon -\delta}|X'|)^s/(3s!)\geq n^{-s\varepsilon-s\delta }|X|^s/(3s!r^s)$ many $s$-sets $S\subseteq X'\subseteq X$ such that the common neighbor of $S$ in $H$ is at least $n^{-s\varepsilon- s\delta }|\mathcal B|/3.$
In other words,
\begin{equation}\label{equ:cd(S,v1)}
cd_{\mathcal G}(S|v_{ 1})\ge  n^{ -s\varepsilon-s\delta}|\mathcal B|/3 \ge n^{r-2-\frac{1}{s-1} -s\varepsilon-(s+1)\delta}/3.
\end{equation}
It suffices to show that $v_{1}$ indeed is a root of such $S$ if $d_{\mathcal G}(S|v_{ 1})\ge n^{r-2-\frac{1}{s-1} -s\varepsilon-(s+1)\delta}/3$.
Write $S=\{w_1,\cdots,w_s\}$ and let $R$ be the set of all roots of $S$.
We know $|R|\leq rt$ from \eqref{equ:roots}.
If $v_1\notin R$, then we can obtain
$$cd_{\mathcal G}(S|v_{1})\le \sum_{x\in R}d_\mathcal{G}(\{w_1,v_1,x\})\le |R|n^{r-3-\frac{1}{s-1} +\varepsilon} \le rt\cdot n^{r-3-\frac{1}{s-1} +\varepsilon}<n^{r-2-\frac{1}{s-1} -s\varepsilon-(s+1)\delta}/3,$$
where the first inequality holds because every set in $\CN_{\mathcal G}(S)$ containing $v_1$ must also contain $w_1$ and a root in $R$,
the second inequality follows from Lemma~\ref{lem: set degree}, and the last inequality holds because $6(s+1)(\varepsilon + \delta)\le 1$.
This contradicts \eqref{equ:cd(S,v1)}. We have finished the proof for the first conclusion.

Next we prove the second conclusion.
Without loss of generality, we assume $j=3$.
Fix a vertex $v\in X\subseteq V_2$.
We define a bipartite graph $H=(A,\mathcal B)$ similarly as above,
where $A:=A(v)$, and $\mathcal B:= N_{\mathcal G}(\{v,v_{1}\})$.
Let $a\in A$ and $B\in \mathcal B$ form an edge in $H$ if and only if $\{v_{1},a\}\cup B\in E(\mathcal G)$.
For each $a \in A$, we have $N_H(a) = \mathcal{B} \cap N_\mathcal{G}(\{v_1,a\})$,
and by \eqref{equ:d(u,v1)}, each $a$ has at least $n^{r-2-\frac{1}{s-1}- \delta}$ neighbors in $\mathcal{B}$.
So $e(H)\ge |A|n^{r-2-\frac{1}{s-1} -\delta} $.
Let $\mathcal R$ be the set of all $(r-3)$-tuples $R\subseteq V_4\cup \cdots\cup V_{r}$ of $\mathcal G$.
For $R\in \mathcal R$, define $\mathcal B_R$ as the set of $(r-2)$-tuples in $\mathcal B$ containing $R$.
So $\mathcal B$ has a partition $\cup_{R\in \mathcal R}\mathcal B_R$.
As $|\mathcal R|\le n^{r-3}$, by averaging there exists some $R^*\in \mathcal R$ such that
$e_H(A,\mathcal B_{R^*})\ge e(H)/n^{r-3}\ge |A|n^{1-\frac{1}{s-1}-\delta}$.
Let $Z:= N_{\mathcal G}(\{v,v_1\}\cup R^*)\subseteq V_3$.
Clearly there is a bijection between $Z$ and $\mathcal B_{R^*}=\{\{z\}\cup R^*| z\in Z\}$.
So we can identify the bipartite subgraph $(A, \mathcal B_{R^*})$ of $H$ as $H'=(A,Z)$,
where $H'$ is defined such that $az\in E(H')$ if and only if $\{v_1,a,z\}\cup R^*\in E(\mathcal{G})$ for $a\in A$ and $z\in Z$.
Then we have $$n^{1-\frac{1}{s-1} + \varepsilon}\geq |Z|\geq e(H')/|A|= e_H(A,\mathcal B_{R^*})/|A|\geq n^{1-\frac{1}{s-1}-\delta},$$
where the first inequality holds due to Lemma~\ref{lem: set degree}.
This further shows that $e(H') \ge |A|n^{1-\frac{1}{s-1}  -\delta} \ge n^{-\varepsilon -\delta}|A||Z|$.
Recalling that $|A|=|A(v)|\geq d_{\mathcal G}(T)/r$, and combining it with $d_{\mathcal G}(T)\geq n^{1-\frac{1}{s-1}- \varepsilon}/\alpha$ and the given restrictions on $\varepsilon, \delta, \alpha$,
we have
$n^{-\varepsilon -\delta}|A|\geq n^{1-\frac{1}{s-1}- 2\varepsilon-\delta}/(\alpha r)\gg s$.
Applying Lemma~\ref{lem: bipartite average}~(1) to $H'=(A,Z)$,
there exists a subset $Y\subseteq A\subseteq N_{\mathcal G}(T)$ of size
$$|Y|\geq n^{-\varepsilon -\delta}|A|/2\ge n^{-\varepsilon -\delta}d_{\mathcal G}(T) /2r\ge n^{1-\frac{1}{s-1}  - 2 \varepsilon -\delta}/2r\alpha$$
such that every $y\in Y$ has
$|N_{\mathcal G}(\{v_1, y\}\cup R^*)\cap Z|=|N_{H'}(y)\cap Z |\ge  n^{ -  \varepsilon -\delta}|Z |/2\ge n^{1-\frac{1}{s-1}  -  \varepsilon-2\delta}/2$.
\end{proof}

Finally, we are ready to show Lemma~\ref{lem: 1root2root3}.

\begin{proof}[\bf Proof of Lemma~\ref{lem: 1root2root3}]
Let $\varepsilon, \delta, \alpha$ be constants and $\mathcal H_1\subseteq\mathcal H\subseteq \mathcal G_1 \subseteq\mathcal G$ be the sequence of
$(\varepsilon, \alpha)$-regular $K_{s,t}^{(r)}$-free hypergraphs given by the statement.
Without loss of generality, we may assume that $V_1 \xrightarrow[\delta]{\mathcal  H_1, \mathcal H}  V_2 \xrightarrow[\delta]{\mathcal  G_1,\mathcal G } V_k$ holds for some $k\neq 2$. We aim to show that $k=1$.

Suppose for a contradiction that $k\notin \{1,2\}$.
Let $T\in \mathcal T_2(\mathcal H_1)$ be an $(r-1)$-tuple with $T\cap V_1=\{v_1\}$.
Since $ V_1 \xrightarrow[\delta]{\mathcal  H_1, \mathcal H } V_2 $,
for any $x\in N_{\mathcal H_1}(T)$, $(T;x)$ is $\delta$-dense on $v_1$ in $\mathcal H$.
Using Lemma~\ref{lem: reduce to bipartite}~(2),
there exist subsets $Y\subseteq N_{\mathcal H}(T)\subseteq V_2$, $Z\subseteq V_k$ and an $(r-3)$-tuple $R\subseteq V(\mathcal G)\backslash (V_1\cup V_2\cup V_k)$ such that
\begin{itemize}
\item $|Y|\ge n^{1-\frac{1}{s-1} - 2 \varepsilon -\delta}/2r\alpha$,
\item  $n^{1-\frac{1}{s-1}  -\delta} \le |Z|\le n^{1-\frac{1}{s-1} + \varepsilon}$, and
\item for each $y_j\in Y$, if we let $T_j=\{v_1, y_j\}\cup R$ and $X_j=N_{\mathcal H}(T_j)\cap Z$, then $|X_j|\ge n^{1-\frac{1}{s-1}-\varepsilon-2\delta}/2$.
\end{itemize}
Let us count the number $m$ of pairs $(S,y)$ such that $y\in Y$ is a root of an $s$-set $S\subseteq Z$ in $\mathcal G$.
Since $Z$ has at most $|Z|^s$ different  $s$-sets, and each $s$-set has at most $rt$ root in $\mathcal G$, we have $m\le rt\cdot |Z|^s$.
Now consider a fixed vertex $y_j \in Y$ (so $y_j\in V_2$).
As $\mathcal H\subseteq \mathcal G_1$,
we see $T_j\in \mathcal T_k(\mathcal H)\subseteq \mathcal T_k(\mathcal G_1)$.
Since $V_2 \xrightarrow[\delta]{\mathcal  G_1,\mathcal G } V_k$,
by Definition~\ref{Def:delta-dense}, every $x\in X_j\subseteq N_{\mathcal H}(T_j)\subseteq N_{\mathcal G_1}(T_j)$ satisfies that
$(T_j;x)$ is $\delta$-dense on $y_j$ in $\mathcal G$.
Since $|X_j|\ge n^{1-\frac{1}{s-1}-\varepsilon-2\delta}/2\gg n^{\varepsilon+\delta}$,
using Lemma~\ref{lem: reduce to bipartite}~(1), we can derive that $X_j\subseteq Z$ contains at least $n^{- s\varepsilon-s\delta }|X_j|^s/(3s!r^s)$ different $s$-sets rooted on $y_j$ in $\mathcal G$.
This shows that $m\ge \sum_{y_j\in Y}n^{- s\varepsilon-s\delta }|X_j|^s/(3s!r^s).$
Putting everything together, we get
$$rt\cdot \left(n^{1-\frac{1}{s-1} +  \varepsilon }\right)^s \ge rt\cdot |Z|^s\ge m
\ge \frac{1}{2r\alpha\cdot 2^s\cdot 3s!r^s} \cdot n^{1-\frac{1}{s-1} - 2 \varepsilon -\delta}\cdot
n^{- s\varepsilon-s\delta } \cdot \left(n^{1-\frac{1}{s-1}-\varepsilon-2\delta}\right)^s.$$
Since $\alpha=o(n^\varepsilon)$,
this implies
\begin{equation}\label{equ:3s(eps+del)}
(3s+3)\varepsilon+(3s+1)\delta\ge 1-1/(s-1)+o(1), \mbox{ where } o(1)\to 0 \mbox{ as } n\to \infty.
\end{equation}
Since $s\ge 3$  and $6(s+1)(\varepsilon + \delta)\le 1$,
we have $(3s+3)\varepsilon+(3s+1)\delta < 3(s+1)(\varepsilon + \delta)\le 1/2 \le  1-1/(s-1)$,
which contradicts \eqref{equ:3s(eps+del)} as $n$ is sufficiently large.
The proof of this lemma is now complete.
\end{proof}

\section{Concluding remarks}
In this paper, we prove that for any odd $r\geq 3$ and any $s\geq 3$, there exists an $\varepsilon_s>0$ such that
$$\ex(n,K_{s,t}^{(r)})=O_{r,s,t} \left(n^{r-\frac{1}{s-1}- \varepsilon}\right).$$
It would be interesting to determine the optimal constant $\varepsilon_s$ for any odd $r\geq 3$.
It is also worth noting that Mubayi and Verstra\"ete \cite{MuVe04} conjectured that the Tur\'an number for 3-uniform hypergraphs satisfies
$\ex(n,K_{s,t}^{(3)})=\Theta_{s,t}\left(n^{3-\frac{2}{s}}\right)$ for any $t\geq s\geq 2$,
which is still open for $s\geq 3$.

As briefly discussed in Subsection~\ref{subsec:sketch},
the proofs of Lemmas \ref{lem: dense subgraph} and \ref{lem: 1root2root3} yield certain rich adjacency structures in dense $K_{s,t}^{(r)}$-free $r$-uniform hypergraphs for even $r\geq 4$.
There structures also align with the construction provided in Section 3 of \cite{BGJS23}.
These observations suggest that perhaps there exists a stability result for $K_{s,t}^{(r)}$ for even $r\geq 4$.

The asymptotics of $f_r(n)$ and $\ex(n,K_{2,2}^{(r)})$ remain intriguing open problems. 
We conclude this paper by mentioning two related conjectures. 
The first conjecture due to F\"uredi \cite{Fu84} states that $\ex(n,K_{2,2}^{(r)})=(1+o(1))\binom{n-1}{r-1}$ for any $r\geq 3$. 
Note that the lower bound $\binom{n-1}{r-1}$ can be achieved by the hypergraph star.
Despise significant progress made in \cite{MuVe04,PiVe09}, this conjecture remains open for any $r\geq 3$.
The second conjecture, posed by Mubayi (see Conjecture 6.2 in~\cite{Mu07}), suggests that the $f_r(n)$-problem is {\it stable} for $r\geq 4$.
This says that for any $r\geq 4$ and $\delta>0$, there exist $\epsilon>0$ and $n_0$ such that any $n$-vertex $r$-uniform hypergraph with at least $(1-\epsilon)\binom{n}{r-1}$ edges, which does not contain four distinct edges $A, B, C, D$ satisfying $A\cup B=C\cup D$ and $A\cap B=C\cap D=\emptyset$, must contain a vertex $v$ belonging to at least $(1-\delta)\binom{n}{r-1}$ edges.

\bigskip
\medskip

{\noindent\bf Acknowledgment.} The authors would like to thank Prof. Hao Huang for helpful discussions.

\bibliographystyle{unsrt}

\end{document}